\documentclass[]{amsart}

\usepackage{stmaryrd}
\usepackage{braket}
\usepackage{enumerate}
\usepackage{proof}
\usepackage{amssymb}
\usepackage{graphicx}

\newtheorem{thm}{Theorem}[section]
\newtheorem{prop}[thm]{Proposition}
\newtheorem{lemma}[thm]{Lemma}
\newtheorem{cor}[thm]{Corollary}
\theoremstyle{definition}

\newcommand\FOT{\ensuremath{\mathrm{FOT}}}
\newcommand\FOTminus{\ensuremath{\mathrm{FOT}^{-}}}
\newcommand\FOTimp{\ensuremath{\FOT(\to)}}
\newcommand\wnot{\mathop{\dot\sim}}

\makeatletter
\newcommand{\wvee}{\mathrel{\vphantom{\lor}\mathpalette\wvee@\relax}}
\newcommand{\wvee@}[2]{\ooalign{\raisebox{.15ex}{\rotatebox[origin=c]{-90}{$\m@th#1\geqslant$}}}}
\makeatother

\newcommand\wexists{\exists^{1}}
\newcommand\wforall{\forall^{1}}
\newcommand\wimp{\mathbin\rightarrowtriangle}

\newcommand\rel{\operatorname{rel}}

\renewcommand\phi\varphi 
\newcommand\dep{\mathop\mathtt{D}}
\newcommand\con{\mathtt{C}}

\renewcommand\iff{\quad\text{iff}\quad}

\newcommand\fotmodels{\models_{\text{FOT}}}
\newcommand\folmodels{\models_{\text{FOL}}}
\newcommand\fotequiv{\equiv_{\text{FOT}}}
\newcommand\folequiv{\equiv_{\text{FOL}}}

\newcommand\FV{\mathop\mathrm{fv}}

\newcommand\M{M}

\title{Intuitionistic Implication in Elementary Team Logics} 
\author{Fredrik Engström}
\address{Department of Philosophy, Linguistics and Theory of Science,
University of Gothenburg, Box 200, SE-405 30 Göteborg, Sweden}
\email{fredrik.engstrom@gu.se}

\author{Juha Kontinen}
\address{Department of Mathematics and Statistics,
University of Helsinki, P.O. Box 68,
FI-00014 University of Helsinki, Finland}
\email{juha.kontinen@helsinki.fi}

\date{\today}

\begin{document}

\begin{abstract} 

The logic FOT is a team-based  logic whose expressive power coincides with
first-order logic at the level of both sentences and open formulas. In
contrast to dependence and independence logics, which can define stronger
second-order team properties, FOT is designed to capture exactly elementary
team properties, modulo the empty team. In this paper we consider two
modifications of FOT. First, we investigate essentially the inclusion atom free fragment
of FOT. Our main result establishes quantifier elimination for the fragment  in  the empty signature. 
Second, we study an extension of FOT by the intuitionistic implication. 
The main conclusion is that adding
this single connective increases the expressive strength so that every second-order sentence can be encoded by an open formula evaluated on the full team. Consequently, validity of formulas is equivalent to validity of full second-order logic.
\end{abstract}

\maketitle

\section{Introduction}

Team semantics is a framework for studying concepts and phenomena that are
inherently collective, involving sets of objects rather than individual
objects. The prime examples of such concepts are, e.g., functional dependence
in database theory and conditional independence  in statistics. The
fundamental ideas underlying team semantics can be traced back to the work of
Hodges~\cite{hodges1997compositional}, while dependence logic and team
semantics as a general logical framework was put forward by 
Väänänen~\cite{Vaananen-book}. Formulas in team-based logics are interpreted over sets of
assignments, called \emph{teams}, instead of single assignments, as is
customary in classical Tarski semantics.

Over the last decade, extensive research has been conducted on the expressive
power and computational complexity of dependence logic and related team-based
logics (see, e.g., \cite{DurandKRV22,KontinenMM22,DurandKV24,10.1145/3771721,Yang13}). Furthermore, the
framework of team semantics has been generalized in various ways
(see, e.g., \cite{Galliani12,DurandHKMV18,HKMV18,10.1145/3834561,HANNULA2022103104}).

We study a variant of dependence logic called FOT~\cite{Kontinen023}. FOT was
specifically designed to match the expressive power of classical first-order
(elementary) logic at both the sentence and formula levels. It is well known
that either of  the tensor disjunction or the existential quantifier of
dependence logic is sufficient to extend the expressive power of a team-based
logic beyond the class of elementary properties (see, e.g., \cite{DurandKRV22}), hence $\FOT$ is defined using an alternate set of primitive operators.

FOT has recently found applications in the study of model theory of
second-order logic~\cite{HYTT26}. Moreover, the extension of 
$\FOT$ with the intuitionistic implication (first defined in  \cite{AbramskyV09}) 
is an extension of the downwards closed logic
InqBT (see \cite{exag032}), which has recently played a central role in
resolving several longstanding open problems in inquisitive first-order
logic~\cite{KontinenCiardelli:26,ciardelli2026}.

The first part of the paper is devoted to the study of the expressive power of a fragment of 
FOT. We prove a quantifier elimination result for this logic in the empty signature and conclude that the corresponding validity problem is decidable.

In the second part, we consider the extension of
FOT by the intuitionistic implication. Although this extension is still
equivalent to first-order logic with respect to sentences, we show that the
expressive power of its open formulas is comparable to that of second-order
logic. As a corollary, we get that the validity problem of the logic is highly undecidable already in the empty signature.

\section{FOT}

Kontinen and Yang take the inclusion atom $\bar{x} \subseteq \bar{y}$ as a
primitive atom in $\FOT$. In the present paper we use a slightly different
presentation. Instead of taking inclusion atoms as primitive, we allow
disequality atoms $\bar x \neq \bar y$ as basic atomic formulas. This change is mostly
notational: the resulting logic is expressively equivalent to the original
presentation of $\FOT$. Indeed, inclusion atoms can be defined in our syntax,
and conversely the basic constructions used here can be simulated in the
language with inclusion atoms. Thus the choice of primitives does not
substantially affect the expressive power of the logic considered below.

Formulas of \FOT\ are given by
\[
\varphi ::= \lambda \mid  \bar x \neq \bar y \mid \wnot\varphi
\mid (\varphi\land\varphi) \mid (\varphi\wvee\varphi)
\mid \wexists x\varphi \mid \wforall x\varphi,
\] 
where $\lambda$ is a first-order atomic formula, including equality: $x=y$,
and $\bar x$ and $\bar y$ are of equal length.


We use the same letter $M$ for a structure and its domain. An assignment over
$M$ is a map from a finite set of variables into $M$. A \emph{team} over $M$
is a set of assignments having a common domain; the empty team is allowed.

If $s$ is an assignment, $x$ is a variable, and $a\in M$, then $s(a/x)$
denotes the assignment which maps $x$ to $a$ and otherwise agrees with $s$.
For a team $X$, put 
\[
X(a/x):=\{s(a/x):s\in X\},
\]
and similarly for tuples $\bar a$ and $\bar x$.
We will also need a global supplementation operation on teams later in the paper: for a function $F\colon X\rightarrow M$, we define
\[
X[F/x]:=\{s(F(s)/x):s\in X\}.
\]
We also write 
\[
X[M/x]:=\{s(a/x):s\in X\textrm{ and }a\in M\}.
\]
The last two operations are used to define semantics for the existential and universal quantifiers in dependence logic.

For a set of variables $V$, define
\[
X\mathop{\upharpoonright}V
:=
\{s\mathop{\upharpoonright}V:s\in X\}.
\]
If $\bar x$ is a tuple of distinct variables, the \emph{full team} over
$\bar x$ is
\[
M^{\bar x}:=\{s:\{\bar x\}\to M\}.
\]
If $X$ is a team let $X(\bar x):=\{s(\bar x):s\in X\}$.

Satisfaction is defined by the following clauses. If $X$ is a team of $M$, then
\[
\begin{aligned}
    M,X \models \lambda
        &\iff M,s \folmodels \lambda \text{ for all }s\in X, \\
     M,X \models \bar x \neq \bar y
        &\iff \text{for all } s \in X, s(\bar x) \neq s(\bar y), \\    
    M,X \models \wnot\varphi
        &\iff X=\emptyset \text{ or } M,X\not\models\varphi, \\
    M,X \models \varphi\wvee\psi
        &\iff M,X\models\varphi \text{ or } M,X\models\psi, \\
    M,X \models \varphi\land\psi
        &\iff M,X\models\varphi \text{ and } M,X\models\psi, \\
    M,X \models \wexists x\varphi
        &\iff M,X(a/x)\models\varphi \text{ for some }a\in M, \\
    M,X \models \wforall x\varphi
        &\iff M,X(a/x)\models\varphi \text{ for all }a\in M,
\end{aligned}
\]
where $X(a/x)=\{s(a/x)\mid s\in X\}$. 

The set of free variables of $\varphi$ is denoted by $\FV(\varphi)$ and is
defined as usual. A formula $\sigma$ is a sentence if
$\FV(\sigma)=\emptyset$. For an $\FOT$-sentence $\sigma$, we write
\[
M\models\sigma
\quad\text{for}\quad
M,\{\emptyset\}\models\sigma,
\]
where $\emptyset$ denotes the unique assignment with empty domain.

\begin{prop}[\cite{kontinen2019logics}]
    \begin{enumerate}
        \item \FOT\ has the empty team property, i.e., $M,\emptyset \models \varphi$ for any model $M$ and any \FOT-formula $\varphi$.
        \item \FOT\ is local, i.e., $M,X \models \varphi$ iff $M,X\mathop\upharpoonright \FV(\varphi) \models \varphi$.
    \end{enumerate}
\end{prop}
In case (2) above, $X\mathop\upharpoonright \FV(\varphi)=\{ s\mathop\upharpoonright \FV(\varphi)\ |\ s\in X   \}$. 

We use the following abbreviations in $\FOT$:
\[
\varphi\wimp\psi:=\wnot\varphi\wvee\psi,
\qquad
\top:=\wforall z\,(z=z),
\qquad
\bot:=\wnot\top.
\]
Thus $\top$ is satisfied by every team, whereas $\bot$ is satisfied exactly
by the empty team.

Observe that $\wnot(x=y)$ is not in general equivalent to $x\neq y$. Also note that, 
the inclusion atom $\bar x\subseteq\bar y$ can be directly defined by the usual inclusion-style definition
\[
\bar x\subseteq\bar y \quad\equiv\quad
\wforall\bar z\bigl(\bar z\neq\bar y \wimp \bar z\neq\bar x\bigr),
\]
where $\bar z$ is a fresh tuple of the same length as $\bar x$ and $\bar y$.

Conversely, tuple disequality is definable using inclusion:
\[
\bar x\neq\bar y\quad\equiv\quad
\wforall\bar v\,\wforall\bar w
\bigl(
\bar v\bar w\subseteq\bar x\bar y
\wimp
\wnot(\bar v=\bar w)
\bigr),
\]
where $\bar v,\bar w$ are fresh tuples of the appropriate lengths.
Thus the two presentations are expressively equivalent.

For a first-order formula $\alpha$, we write
$\alpha^\ast$ for the $\FOT$-formula obtained from $\alpha$ by replacing
$\neg,\vee,\exists,\forall$ by $\wnot,\wvee,\wexists,\wforall$, respectively.
Conversely, if $\varphi$ is a \FOT-formula, we write $\varphi^\sharp$ for the
first-order formula we obtain by replacing $\wnot,\wvee,\wexists,\wforall$ by
$\neg,\vee,\exists,\forall$, respectively. For primitive tuple disequality, put
\[
(\bar x\neq\bar y)^\sharp
:=
\neg\bigwedge_{i=1}^{|\bar x|}x_i=y_i.
\]

\begin{prop} 
    For any structure $M$, any $\FOT$-formula $\varphi$, and
    any assignment $s$ with domain containing the free variables of
    $\varphi$,
\[
M,s \folmodels \varphi^\sharp
\quad\text{iff}\quad
M,\{s\} \fotmodels \varphi.
\]
\end{prop}

\begin{proof}
The proof is by induction on $\varphi$. The atomic cases are immediate. The induction 
cases follow directly from the corresponding
truth conditions on singleton teams. 
\end{proof}

\begin{cor}
For any $\FOT$-sentence $\sigma$ and any model $M$,
\[
M \folmodels \sigma^\sharp
\quad\text{iff}\quad
M \fotmodels \sigma.
\]
\end{cor}

We recall the expressive completeness result for $\FOT$. If
$\bar x$ is a tuple of distinct variables and $R$ is a fresh
$n$-ary relation symbol, write
\[
\rel_{\bar x}(X):=\{s(\bar x):s\in X\}.
\]
\begin{thm}[\cite{Kontinen023}]

For every $\FOT$-formula $\varphi(\bar x)$ there is a first-order sentence
$\tau_\varphi(R)$ such that
\[
M,X\models\varphi \iff (M,\rel_{\bar x}(X))\models\tau_\varphi(R).
\]
Conversely, for every first-order sentence $\gamma(R)$ such that
\[
    (M,\emptyset)\models\gamma(R)
\]
for every structure $M$, there exists an $\FOT$-formula
$\varphi_\gamma(\bar x)$ such that, for every structure $M$ and every
team $X$ over $\bar x$,
\[
    M,X\models\varphi_\gamma \iff (M,\rel_{\bar x}(X))\models\gamma(R).
\]
\end{thm} 
Thus $\FOT$ captures exactly the elementary team properties, modulo the
empty-team property.

\section{\FOTminus: Losing $\neq$}

We now turn from sentences to formulas with free variables. For sentences, the weak
connectives of \FOT\ behave much like their first-order counterparts. In fact, for sentences \FOTminus is still equivalent to first-order logic. Over open
formulas, however, the team-semantical interpretation becomes visible: a formula is
evaluated not at a single assignment, but at a whole team of assignments. This means
that first-order equivalence of the corresponding formulas need not determine
equivalence in team semantics.

To make this precise, write $\varphi \fotequiv \psi$ for \FOT-formulas $\varphi$ and
$\psi$ if, for every model $M$ and every team $X$,
\[
M,X \models \varphi \quad\text{iff}\quad M,X \models \psi.
\]

Let $\FOTminus$ be the fragment of \FOT\ obtained by removing primitive
disequality atoms. Thus, formulas of $\FOTminus$ are given by
\[
\varphi ::= \lambda \mid \wnot\varphi
\mid (\varphi\land\varphi) \mid (\varphi\wvee\varphi)
\mid \wexists x\varphi \mid \wforall x\varphi,
\]
where $\lambda$ is a first-order atomic formula, including equality atoms.
Observe that this is exactly the inclusion-free fragment of $\FOT$ as defined
in \cite{kontinen2019logics}. 

The following proposition shows that, already for this fragment, the passage
from pointwise truth to team truth can destroy equivalences that hold
classically.

\begin{prop}
There are $\FOTminus$-formulas $\varphi$ and $\psi$ such that
$\varphi^\sharp \folequiv \psi^\sharp$, but $\varphi \not\fotequiv \psi$.
\end{prop}

\begin{proof}
Let
\[
\sigma_2 :=
\wexists u \wexists v
\bigl(\wnot(u=v) \land \wforall w (w=u \wvee w=v)\bigr).
\]
On singleton teams, $\sigma_2$ expresses that the domain has exactly two
elements. 

Now define
\[
\varphi :=
\sigma_2
\land \wnot(x=y)
\land \wnot(y=z)
\land \wnot(x=z).
\]
Classically, the formula $\varphi^\sharp$ is unsatisfiable: in a two-element
model no assignment can give pairwise distinct values to $x,y,z$. 

Thus, if $\psi$ is $\wnot(x=x)$ then $\varphi^\sharp$ is equivalent to
$\psi^\sharp$. However, $\varphi$ is not equivalent to $\psi$ in team
semantics: Let $M=\{0,1\}$ and let $X$ be the team
\[
\begin{array}{ccc}
x & y & z \\ \hline
0 & 1 & 0 \\
1 & 0 & 0
\end{array}
\]
over $M$. Since $M$ has exactly two elements, we have $M,X \models \sigma_2$.
Moreover,
\[
M,X \models \wnot(x=y), \qquad
M,X \models \wnot(y=z), \qquad
M,X \models \wnot(x=z),
\]
because each of the equalities $x=y$, $y=z$, and $x=z$ fails on at least one
assignment in $X$. Hence $M,X \models \varphi$.

On the other hand, $M,X \not\models \psi$, since $x=x$ is satisfied by every
assignment in $X$, and $X$ is nonempty. Therefore
\[
\varphi \not\fotequiv \psi,
\]
even though $\varphi^\sharp \folequiv \psi^\sharp$.
\end{proof}

Thus, the team-semantical denotation of a $\FOTminus$-formula is not determined
by the ordinary first-order denotation of its translation.

In the following two lemmas, $\bar x$ is a tuple of distinct variables
disjoint from $\bar y$ and bound variables are renamed
when necessary to avoid clashes.

\begin{lemma} 
Let $X$ and $Y$ be two teams defined over variables $\{y_1, \ldots, y_n\}$, $n \geq 1$,
and over a model $M$ such that 
\[ M,X \models y_i=y_j \iff  M,Y \models y_i=y_j \] 
for all $i,j \in \{1,\ldots,n\}$. Suppose also that $|X(y_i)| \geq 2$ and 
$|Y(y_i)| \geq 2$ for all $1 \leq i \leq n$. 

Then for all $\bar a \in M$, $M,X
[\bar a/\bar x] \models \varphi$ iff $M,Y[\bar a/\bar x] \models \varphi$ for
all \FOTminus-formulas $\varphi$ in the empty signature.
\end{lemma}
\begin{proof} 
By induction over the structure of $\varphi$. The three base cases $y_i=y_j$,
$x_i=x_j$ and $x_i = y_j$ follow directly from the assumptions. The
inductive cases, $\land$, $\wvee$, $\wnot$, $\wexists$ and $\wforall$, are
all straightforward.
\end{proof}

We can remove the assumptions on the sizes of the projections by recording
which variables are constant. Let
\[
\con(x):=\wexists z\,(z=x),
\]
where $z$ is fresh. Then, for every team $X$,
\[
M,X\models\con(x) \iff |X(x)|\leq1.
\]
In particular, if $X$ is nonempty, then $\con(x)$ holds exactly when $x$ has
a unique constant value throughout $X$.

\begin{lemma}\label{lemma:eq-con-invariance}
    Let $M$ be a model, and $X$ and $Y$ be two nonempty teams with domain $\{y_1, \ldots, y_n\}$ such that 
    \begin{align*} 
    M,X \models y_i=y_j &\iff  M,Y \models y_i=y_j, \text{ and} \\
    M,X \models \con(y_i) &\iff  M,Y \models \con(y_i),
    \end{align*}
    for all $i,j \in \{1,\ldots,n\}$. Then there is a bijection $f: M \to M$ such that for all $\bar a \in M$,
    \[ M,X[\bar a/\bar x] \models \varphi \iff M,Y[f(\bar a)/\bar x] \models \varphi\] for all \FOTminus-formulas $\varphi$ in the empty signature. 
\end{lemma}
\begin{proof}

    For every $i$ such that $M,X \models \con(y_i)$ let $c_i$ be the unique element of $X(y_i)$ and $d_i$ be the unique element of $Y(y_i)$. Let $f: M \to M$ be any bijection such that $f(c_i)=d_i$. There are such bijections since $c_i = c_j$ iff $d_i=d_j$.

    We prove that for any \FOTminus-formula $\varphi$ in the empty signature and any $\bar a \in M$ we have 
    \[ M,X[\bar a/\bar x] \models \varphi \iff M,Y[f(\bar a)/\bar x] \models \varphi\]
    by induction over the structure of $\varphi$:
    
    $M,X[\bar a/\bar x] \models y_i=y_j$ iff $M,X \models y_i=y_j$ iff $M,Y \models y_i=y_j$ iff $M,Y[f(\bar a)/\bar x] \models y_i=y_j$.

If $y_i$ is nonconstant, then $M,X[\bar a/\bar x]\not\models y_i=x_j$ and $M,Y
[f(\bar a)/\bar x]\not\models y_i=x_j$. Otherwise,
$M,X[\bar a/\bar x] \models y_i=x_j$ iff $M,X \models \con(y_i)$ and $c_i=a_j$. This holds iff $d_i=f(a_j)$ iff $M,Y[f(\bar a)/\bar x] \models y_i=x_j$.
    
    $M,X[\bar a /\bar x] \models x_i=x_j$ iff $a_i=a_j$ iff $f(a_i)=f(a_j)$ iff $M,Y[f(\bar a)/\bar x] \models x_i=x_j$. 
     
    The cases of $\land, \wvee$ and $\wnot$ are straightforward. 

    $M,X[\bar a/\bar x] \models \wexists z \varphi$ iff there is $e \in M$, $M,X[\bar a/\bar x][e/z] \models \varphi$ iff there is $e \in M$, $M,Y[f(\bar a)/\bar x][f(e)/z] \models \varphi$ iff there is $e' \in M$, 
    $M,Y[f(\bar a)/\bar x][e'/z] \models \varphi$ iff $M,Y[f(\bar a)/\bar x] \models \wexists z \varphi$. 

    The universal-quantifier case is analogous.
\end{proof}

We can now obtain quantifier elimination over every fixed structure. Given a
structure $M$, write
\[
\varphi\equiv_M\psi
\]
if, for every team $X$ over $M$ whose domain contains the free variables of
$\varphi$ and $\psi$,
\[
M,X\models\varphi \iff M,X\models\psi.
\]

\begin{prop}
\label{prop:fixed-model-qe}
Let $M$ be a structure and let $\varphi(\bar x)$ be an
$\FOTminus$-formula in the empty signature. Then there is an $\FOTminus$-formula $\psi(\bar x)$, 
constructed from literals of the forms
\[
x_i=x_j,\qquad \wnot(x_i=x_j),\qquad
\con(x_i),\qquad \wnot\con(x_i)
\]
using $\land$ and $\wvee$, together with $\bot$ if necessary, such that
\[
\varphi\equiv_M\psi.
\]
\end{prop}

\begin{proof}
Let $X$ be a nonempty team over
$\bar x=(x_1,\ldots,x_n)$. Define the complete
\emph{equality--constancy type} of $X$ over $M$ by
\[
\operatorname{tp}_M(X)
:=
\bigwedge_{1\leq i,j\leq n}\ell_{ij}^X
\land
\bigwedge_{1\leq i\leq n}\kappa_i^X,
\]
where
\[
\ell_{ij}^X
:=
\begin{cases}
x_i=x_j,
   &\text{if }M,X\models x_i=x_j,\\
\wnot(x_i=x_j),
   &\text{if }M,X\not\models x_i=x_j,
\end{cases}
\]
and
\[
\kappa_i^X
:=
\begin{cases}
\con(x_i),
   &\text{if }M,X\models\con(x_i),\\
\wnot\con(x_i),
   &\text{if }M,X\not\models\con(x_i).
\end{cases}
\]
Because $X$ is nonempty, exactly one literal is selected in each of these
pairs. In particular,
\[
M,X\models\operatorname{tp}_M(X).
\]

There are only finitely many syntactically possible equality--constancy
types over $\bar x$. Let
\[
\mathcal T_{\varphi,M}
:=
\bigl\{
\operatorname{tp}_M(Y):
Y\neq\emptyset
\text{ and }
M,Y\models\varphi
\bigr\}.
\]
Enumerate the distinct members of this finite set as
$\tau_1,\ldots,\tau_m$. Define
\[
\psi(\bar x):=
\begin{cases}
\bot,
   &\text{if }m=0,\\
\tau_1\wvee\cdots\wvee\tau_m,
   &\text{if }m>0.
\end{cases}
\]

We show that $\varphi\equiv_M\psi$. First let $X=\emptyset$. Then
\[
M,X\models\varphi
\qquad\text{and}\qquad
M,X\models\psi
\]
by the empty-team property.

Now suppose that $X$ is nonempty. If $M,X\models\varphi$, then
$\operatorname{tp}_M(X)\in\mathcal T_{\varphi,M}$. Hence
$M,X\models\psi$.

Conversely, suppose that $M,X\models\psi$. Since $X$ is nonempty,
$M,X\not\models\bot$, so $m>0$ and
\[
M,X\models\tau_j
\]
for some $j\leq m$. By the definition of $\mathcal T_{\varphi,M}$, there is
a nonempty team $Y$ such that
\[
M,Y\models\varphi
\qquad\text{and}\qquad
\tau_j=\operatorname{tp}_M(Y).
\]
Since $X$ is nonempty and satisfies the complete type of $Y$, the teams $X$
and $Y$ agree on all equality and constancy data:
\begin{align*}
M,X\models x_i=x_j
&\quad\Longleftrightarrow\quad
M,Y\models x_i=x_j,\\
M,X\models\con(x_i)
&\quad\Longleftrightarrow\quad
M,Y\models\con(x_i).
\end{align*}
Lemma~\ref{lemma:eq-con-invariance} therefore gives
\[
M,X\models\varphi
\quad\Longleftrightarrow\quad
M,Y\models\varphi.
\]
The right-hand side holds by the choice of $Y$, and hence
$M,X\models\varphi$.
Thus, $\varphi\equiv_M\psi$.
\end{proof}

The formula supplied by Proposition~\ref{prop:fixed-model-qe} may depend on
the structure $M$. Since the size of the domain is fixed, cardinality
conditions need not be recorded explicitly. To obtain a normal form that
works uniformly over all structures, we additionally use the sentences
$\sigma_{\geq n}$ expressing finite lower bounds on $|M|$.

For formulas $\rho_i$, $i\in I$, write
\[
\bigvee_{i\in I}^{\mathrm w}\rho_i
\]
for any fixed finite iteration of the weak disjunction $\wvee$. If
$I=\emptyset$, this expression is understood as $\bot$.

\begin{prop}[Uniform quantifier elimination]
\label{prop:uniform-qe}
Let $\varphi(\bar x)$ be an $\FOTminus$-formula in the empty signature.
Then there is an $\FOTminus$-formula $\psi(\bar x)$ which is a finite
combination, using $\wnot$, $\land$, and $\wvee$, of formulas of the forms
\[
x_i=x_j,\qquad \con(x_i),\qquad \sigma_{\geq n},
\]
where $n\geq1$ and
\[
\sigma_{\geq n}
:=
\wexists z_1\cdots\wexists z_n
\bigwedge_{1\leq i<j\leq n}\wnot(z_i=z_j),
\]
such that 
\[
\varphi(\bar x)\fotequiv\psi(\bar x).
\]
\end{prop}
\begin{proof}
Let $\mathcal B(\bar x)$ be the class of all finite combinations using $\wnot$, $\land$, and $\wvee$ of
formulas of the forms
\[
x_i=x_j,\qquad \con(x_i),\qquad \sigma_{\geq n},
\]
whose free variables are among $\bar x$. We prove by induction on formulas
that every $\FOTminus$-formula whose free variables are among $\bar x$ is
equivalent to a formula in $\mathcal B(\bar x)$.

The atomic case is immediate, since the only atoms in the empty signature
are equalities. The induction steps for $\land$, $\wvee$, and $\wnot$ follow
directly from the definition of $\mathcal B$.

It remains to prove closure under the weak quantifiers. Suppose that
\[
\theta(\bar x,y)\in\mathcal B(\bar x,y).
\]
We first eliminate $\wexists y$.

We may assume that $\theta$ is in disjunctive
normal form. Furthermore, weak
existential quantification distributes over finite weak disjunction.
It is thus enough to eliminate $\wexists y$ from a conjunction
$\delta(\bar x,y)$ of literals of the forms
\[
u=v,\quad \wnot(u=v),\quad
\con(u),\quad \wnot\con(u),\quad
\sigma_{\geq n},\quad\wnot\sigma_{\geq n},
\]
where $u,v$ are among $\bar x,y$. In fact, by standard considerations,
we may assume that every literal contains $y$.

The literals $y=y$ and $\con(y)$ are automatically satisfied after choosing
a constant value for $y$. By contrast, $\wnot(y=y)$ and
$\wnot\con(y)$ are false on every nonempty team obtained by such a choice.
Consequently, if either of the latter two literals occurs in $\delta$, then
\[
\wexists y\,\delta\fotequiv\bot.
\]

Suppose next that a positive equality $y=x_j$ occurs in $\delta$. On a
nonempty team, a constant witness for $y$ satisfies $y=x_j$ exactly when
$x_j$ is constant, in which case the witness is forced to be the unique
value of $x_j$. It follows, also on the empty team, that
\[
\wexists y\,\delta
\fotequiv
\con(x_j)\land
\delta[x_j/y].
\]
The right-hand side belongs to $\mathcal B(\bar x)$.

We are left with the case
\[
\delta(\bar x,y)
=
\bigwedge_{i\in N}\wnot(y=x_i)
\]
for some $N\subseteq\{1,\ldots,k\}$, where
$\bar x=(x_1,\ldots,x_k)$. The set $N$ is allowed to be empty.

A complete equality--constancy type is realizable if and only if
its positive equalities define an equivalence relation and,
whenever $x_i=x_j$ belongs to the type, $\con(x_i)$ belongs to
the type if and only if $\con(x_j)$ does.
Every such type in $k$ variables is realizable in every structure
of size at least $\max(k,2)$: assign distinct values to the
constant classes and let the remaining classes vary independently.

Let $\Pi_k$ be the finite set of all complete equality--constancy types over
$\bar x$ which are realized by a nonempty team in some structure. For
$\pi\in\Pi_k$, let $\alpha_\pi(\bar x)$ denote the conjunction defining
$\pi$. Thus every nonempty team over $\bar x$ satisfies exactly one formula
$\alpha_\pi$.

For every $\pi\in\Pi_k$, define an equivalence relation $\sim_\pi$ on
$\{1,\ldots,k\}$ by
\[
i\sim_\pi j
\quad\Longleftrightarrow\quad
x_i=x_j\text{ occurs positively in }\alpha_\pi.
\]
Let
\[
C_\pi
:=
\{i\in N:\con(x_i)\text{ occurs positively in }\alpha_\pi\}
\]
and put
\[
c_\pi:=|C_\pi/{\sim_\pi}|.
\]
If a nonempty team $X$ satisfies $\alpha_\pi$, then $c_\pi$ is precisely the
number of distinct constant values among the variables $x_i$ with $i\in N$.

For such a team, a value $a\in M$ fails as a witness for $y$ exactly when
$a$ is the constant value of one of these variables. Nonconstant variables
$x_i$ exclude no value, because no constant choice of $y$ makes
$y=x_i$ true throughout the team. Hence
\[
M,X\models
\wexists y\bigwedge_{i\in N}\wnot(y=x_i)
\quad\Longleftrightarrow\quad
|M|\geq c_\pi+1.
\]
For every nonempty team $X$,
\[
M,X\models\sigma_{\geq n}
\quad\Longleftrightarrow\quad
|M|\geq n.
\]
We therefore obtain
\[
\wexists y\,\delta \fotequiv
\bigvee_{\pi\in\Pi_k}^{\mathrm w}
\left(
\alpha_\pi(\bar x)\land\sigma_{\geq c_\pi+1}
\right).
\] 

For nonempty teams this follows from the preceding counting argument and the
fact that every such team realizes exactly one complete type. On the empty
team, both sides  hold by the empty-team property. This is clearly in $\mathcal B(\bar x)$.
This proves that $\mathcal B$ is
closed under $\wexists$.

Finally,
\[
\wforall y\,\theta
\fotequiv
\wnot\wexists y\,\wnot\theta.
\]
Thus $\mathcal B$ is also closed under $\wforall$.

It follows by induction that every $\FOTminus$-formula in the empty
signature is equivalent to a formula of the required form.
\end{proof}

It follows from these results that $\FOTminus$ is strictly weaker than
$\FOT$ at the level of open formulas. Recall that the inclusion atom 
$\bar x\subseteq \bar y$ is definable in $\FOT$. In contrast, it is not expressible
in $\FOTminus$.

\begin{prop}
There is no $\FOTminus$-formula $\varphi(x,y)$ in the empty signature such
that, for every model $M$ and every team $X$ over $\{x,y\}$,
$$
M,X\models\varphi(x,y) \iff M,X\models x\subseteq y.
$$
\end{prop}

\begin{proof}
Let $M=\{0,1,2\}$. Consider the two teams $X$ and $Y$ over
$\{x,y\}$ determined by
\begin{gather*}
X(x,y)=\{0,1\} \times \{0,1,2\}, \text{ and } \\
Y(x,y)=\{0,1\} \times \{0,2\}.
\end{gather*}
The teams $X$ and $Y$ have the same equality--constancy type. Hence, by the previous result, $X$ and $Y$ satisfy
exactly the same $\FOTminus$-formulas in the empty signature.

On the other hand,
\begin{gather*}
M,X\models x\subseteq y, \text{ but }\\
M,Y\not\models x\subseteq y.
\end{gather*}
Consequently no $\FOTminus$-formula can define the inclusion atom.
\end{proof}

As a consequence of the quantifier-elimination theorem, we obtain the
following decidability result.

\begin{cor}
For $\FOTminus$ in the empty signature, the following problems are
decidable:
\begin{enumerate}
    \item whether a formula is satisfiable by some nonempty team;
    \item whether a formula is valid, i.e.\ satisfied by every team over every
    model; and
    \item whether two formulas are equivalent.
\end{enumerate}
\end{cor}

\begin{proof}
The proof of the quantifier-elimination theorem is effective. Hence, given an
$\FOTminus$-formula $\varphi(\bar x)$, we can effectively construct an
equivalent Boolean combination of formulas of the form
$$
x_i=x_j,\qquad \con(x_i),\qquad \sigma_{\geq n}.
$$

Let $k=|\bar x|$, and let $N$ be the largest cardinality threshold
occurring in the normal form, taking $N=1$ if none occurs.
Put $K=\max(N,k,2)$. By the realizability criterion above, all
structures of size at least $K$ realize exactly the same
equality--constancy types. They also agree on all cardinality tests
in the normal form. It therefore suffices to consider domain sizes
$1,\ldots,K$, with $K$ representing all larger sizes, and the
finitely many types realizable at each size. Evaluating the normal
form on these cases decides satisfiability by a nonempty team and
validity.

For equivalence, given formulas $\varphi$ and $\psi$, compute their
normal forms and compare their truth values on all realizable
equality--constancy types and all relevant cardinality cases. Since there are
only finitely many such cases, this is decidable as well. The empty team
requires no separate test, since every $\FOTminus$-formula is satisfied by
the empty team.
\end{proof}

Thus the loss of disequality has a substantial effect on expressive power.
While $\FOT$ can compare the sets of values assumed by variables through
inclusion, $\FOTminus$ cannot make even this basic comparison. Its formulas
in the empty signature are limited to the equality and constancy information
described by the quantifier-elimination theorem, together with cardinality
properties of the underlying model. The full logic $\FOT$ is
strictly stronger, although still rather weak in expressive terms. As we
shall see in the next section, however, adding intuitionistic implication
changes the situation dramatically and yields a logic of much greater
expressive power.

\section{Adding intuitionistic implication}

We now consider the extension of FOT obtained by adding intuitionistic
implication. We denote this logic by $\FOTimp$. Thus, in addition
to the formation rules for FOT, we allow formulas
\[
    \varphi \to \psi.
\]
The satisfaction clause is the usual intuitionistic implication clause from team semantics:
\[
    M,X\models \varphi\to\psi
    \iff
    \text{for every }Y\subseteq X,\text{ if }M,Y\models\varphi,\text{ then }M,Y\models\psi.
\]

We shall use the following abbreviations. Let
\[
    \top := \wforall z( z=z)
\]
and
\[
    \bot := \wnot\top.
\]
Thus $\top$ is true on every team, and $\bot$ is true exactly on the empty team.

The connective $\to$ makes it possible to quantify over subteams. For instance,
\[
    \Box\varphi := \top\to\varphi
\]
satisfies
\[
    M,X\models \Box\varphi
    \iff
    M,Y\models\varphi\text{ for every }Y\subseteq X.
\]

It is worth observing that, once intuitionistic implication is added to
$\FOTminus$, primitive disequality becomes definable. Thus, $\FOTminus
(\to)$ has the same strength as $\FOTimp$.

\begin{prop}
In $\FOTminus(\to)$, primitive tuple disequality $\bar x\neq\bar y
$ is definable by
$$
\delta(\bar x,\bar y):=\left(\bigwedge_i x_i=y_i\right)\to\bot.
$$
\end{prop}
\begin{proof}
Suppose first that
$$
M,X\models\delta(\bar x,\bar y)
$$
for some model $M$ and team $X$.
If there were some $s\in X$ such that
$$
s(\bar x)=s(\bar y),
$$
then the singleton team $\{s\}$ would satisfy
$$
\bigwedge_i x_i=y_i
$$
but would not satisfy $\bot$. Since $\{s\}\subseteq X$, this shows that 
$$
M,X\models\bar x\neq\bar y.
$$

Conversely, suppose that
$$
M,X\models\bar x\neq\bar y.
$$
Let $Y\subseteq X$ and assume
$$
M,Y\models\bigwedge_i x_i=y_i.
$$
If $Y$ were nonempty, this would contradict 
$M,X\models\bar x\neq\bar y$. Hence $Y=\emptyset$, and therefore
$M,Y\models\bot$.
\end{proof}

Let us state some easy observations. For this we extend the $\cdot^\sharp$ translation to $\FOTimp$ in the obvious way: 
\[
    (\varphi \to \psi)^\sharp \text{ is } \varphi^\sharp \to \psi^\sharp.
\]

\begin{prop}\label{prop4.2}
\begin{enumerate}
    \item  $\FOTimp$ has the empty team property. That is, for every
        model $M$ and every $\FOTimp$-formula $\varphi$, $M,\emptyset\models\varphi$.
    \item $\FOTimp$ is local. That is, for every model $M$, every
      team $X$, and every $\FOTimp$-formula $\varphi$, 
      $M,X\models\varphi  \text{ iff }  M,X\mathop\upharpoonright\FV(\varphi)\models\varphi$.
    \item  For every model $M$, every assignment $s$, and every $\mathrm
        {FOT}(\to)$-formula $\varphi$, $M,s\folmodels \varphi^\sharp \text{ iff } M,\{s\}\models\varphi$.
        \item\label{aaaa} For every $\FOTimp$-sentence $\sigma$ and every model $M$, 
        $M\folmodels \sigma^\sharp  \text{ iff }  M\models\sigma$.
\item\label{cccc} For every $\FOTimp$-formula $\varphi$ there exists a sentence $\varphi^*$ of second-order logic such that for all $M$ and teams $X$ over $\FV(\varphi)=\bar x$,
$$M,X\models \varphi \iff (M,\rel(X))\models \varphi^*,$$
 where $\rel(X)=\{s(\bar x )\,:\, s\in X \}$       \end{enumerate}
\end{prop}
\begin{proof}
The empty team property, locality, and singleton correspondence
follow by structural induction. For locality of implication,
a subteam of the projection is lifted to its preimage in the
original team. On a singleton team, the only subteams are the
singleton itself and the empty team, so implication has its
classical truth condition. The sentence claim follows by
evaluating on $\{\emptyset\}$.

The last item follows from the standard effective translation
into second-order logic; see \cite{KontinenN11}. In particular,
the implication clause quantifies universally over subrelations
of the relation representing the team.
\end{proof}

Thus, intuitionistic implication does not add expressive power at the level of
sentences, when formulas are evaluated on singleton teams. Its effect is
instead visible at the level of open formulas and arbitrary teams.

It is also immediate that, in $\FOTimp$, we can define functional dependence
$\dep(x_0,\ldots, x_k,y)$ as
$$\con(x_0) \land \ldots \land \con(x_k) \to \con(y).$$

Below we say that  a function symbol occurs \emph{normally} in a formula $\theta$ if all of its occurrences are of the form $f_i(\bar{z_i})$ with a unique tuple $\bar{z_i}$  of pairwise distinct variables.

\begin{lemma}[Flat translation of quantifier-free formulas]\label{lemma:flat}
Let $\theta(\bar x)$ be a quantifier-free first-order formula in which $f_1,\ldots, f_n$ occur normally. Then there is an
$\FOTimp$-formula $\theta^\flat(\bar x\bar y)$, where $\bar y =y_1,\ldots, y_n$ are fresh variables such that, for every model $M$ and every team
$X$ over $\bar x$,
\[
M,X^*\models \theta^\flat \iff (M,\bar f),s\models \theta \text{ for every }s\in X,
\]
where $X^*=X[\bar F/\bar y]$ and $F_i(s)=f_i(s(\bar z_i))$.
\end{lemma}

\begin{proof}
We define $\theta^\flat$ by induction on $\theta$.
If $\theta$ is atomic, let $\theta^\flat:=\theta[\bar y / \overline{f(\bar  z)}]$. Let
\begin{align*}
(\alpha\land\beta)^\flat &:= \alpha^\flat\land\beta^\flat,\\ 
(\neg\alpha)^\flat &:= \alpha^\flat \to \bot,\\ 
(\alpha\lor\beta)^\flat &:= (\alpha^\flat\to\beta^\flat)\to\beta^\flat.
\end{align*}

For $s\in X$, write $s^*$ for its extension satisfying
$s^*(y_i)=f_i(s(\bar z_i))$. Every subteam of $X^*$ is of the
form $Y^*=\{s^*:s\in Y\}$ for some $Y\subseteq X$.

We prove the lemma by induction on $\theta$.
The atomic case follows by substitution and the case of conjunction
is immediate. For negation,
\[
M,X^*\models\alpha^\flat\to\bot
\]
iff every subteam $Y^*\subseteq X^*$ satisfying $\alpha^\flat$
is empty. By the induction hypothesis, this is equivalent to saying
that no singleton $\{s^*\}\subseteq X^*$ satisfies $\alpha^\flat$,
which is equivalent to
\[
(M,\bar f),s\not\models\alpha
\quad\text{for every }s\in X,
\]
and hence to
\[
(M,\bar f),s\models\neg\alpha
\quad\text{for every }s\in X.
\]

For disjunction, suppose first that every $s\in X$ satisfies
$\alpha\lor\beta$ in $(M,\bar f)$. We show that
$M,X^*\models(\alpha\lor\beta)^\flat$.
Let $Y\subseteq X$ and assume
\[
M,Y^*\models\alpha^\flat\to\beta^\flat.
\]
We show $M,Y^*\models\beta^\flat$. Let $s\in Y$.
If $(M,\bar f),s\models\beta$, then
$M,\{s^*\}\models\beta^\flat$ by the induction hypothesis.
Otherwise, $(M,\bar f),s\models\alpha$, so
$M,\{s^*\}\models\alpha^\flat$, and the assumed implication
again gives $M,\{s^*\}\models\beta^\flat$.
Thus every $s\in Y$ satisfies $\beta$ in $(M,\bar f)$, and hence
$M,Y^*\models\beta^\flat$ by the induction hypothesis.

Conversely, suppose
$M,X^*\models(\alpha\lor\beta)^\flat$.
Let $s\in X$. We show that
$(M,\bar f),s\models\alpha\lor\beta$.
If $(M,\bar f),s\models\beta$, this is clear.
Otherwise, $M,\{s^*\}\not\models\beta^\flat$ by the induction
hypothesis. Since
$M,X^*\models(\alpha^\flat\to\beta^\flat)\to\beta^\flat$,
it follows that
\[
M,\{s^*\}\not\models\alpha^\flat\to\beta^\flat.
\]
Thus there is a subteam $Z\subseteq\{s^*\}$ such that
$M,Z\models\alpha^\flat$ and $M,Z\not\models\beta^\flat$.
By the empty team property, this subteam must be $Z=\{s^*\}$.
Hence $M,\{s^*\}\models\alpha^\flat$, so
$(M,\bar f),s\models\alpha$ by the induction hypothesis.
Therefore $(M,\bar f),s\models\alpha\lor\beta$.

This completes the induction.
\end{proof}

\begin{thm}\label{thm:sol}
    Let $\Phi$ be a second-order sentence in the signature $\tau$. Then there is an open $\FOTimp$-formula
    $\Phi^\circ(\bar r)$ such that, for every model $M$,
    \[
        M\models\Phi
        \iff
        M,M^{\bar r}\models\Phi^\circ(\bar r),
    \]
    where $M^{\bar r}$ denotes the full team over $\bar r$.
\end{thm}

\newcommand{\Tot}{\mathrm{Tot}}
\newcommand{\Code}{\mathrm{Code}}

\begin{proof}
We may assume, using the prenex normal form of second-order logic and dummy function variables, that $\Phi$ has the following form:
\[
\forall f_1\,\exists f_2 \,\ldots\, \forall f_{2n-1}\,\exists f_{2n} \,\forall \bar x\,\theta,
\]
where $\theta$ is quantifier-free. We may moreover assume, after the usual
flattening and renaming of function terms, that every occurrence of a function
symbol in $\theta$ is normal, i.e., has the form
\[
f(\bar z_f) 
\]
where the tuples $\bar z_f$ consist of pairwise distinct variables from
$\bar x$. For each such term $f(\bar z_f)$ we introduce a fresh variable. 

For the sake of the induction argument we prove the slightly stronger statement:
\begin{quote}
Let $\Phi$ be a  second-order $\Pi^1_{2n}$-formula in the signature $\tau \cup \set{g_1,\ldots,g_k}$ of the form described above:
\[
\forall f_1\,\exists f_2 \,\ldots\, \forall f_{2n-1}\,\exists f_{2n} \,\forall \bar x\,\theta.
\]
Then there is an \FOTimp-formula 
\[
\Phi^\circ(\bar x,y_1,\ldots,y_{2n},u_1,\ldots,u_k)
\]
such that 
$$
(M,\bar g) \models \Phi \quad\text{\iff}\quad 
M,M^{\bar x\bar y}[\bar G/\bar u] \models \Phi^\circ(\bar x, \bar y,\bar u),
$$
where $G(s)=g(s(\bar z_g))$.
\end{quote}

The idea of the translation is that each quantified function $f_i$ and the corresponding term $f_i(\bar z_{f_i})$ is represented by the variable $y_i$ whereas the function symbols $g_i$ occurring free are encoded by the variables $u_i$. It is important to note that on the team semantics side only the interpretations of $g_i$'s are fixed in the team to correspond to their interpretations in $(M,g_1,\ldots,g_k)$.

For the base case $n=0$, we let $\Phi^\circ$ be $\theta^\flat$, the flat
$\FOTimp$-translation of the quantifier-free matrix $\theta$ of $\Phi$.
Lemma \ref{lemma:flat} shows that it has the right properties.

For the induction step, suppose that
\[
\Phi=\forall f_1\,\exists f_2\,\Psi,
\]
where $\Psi\in\Pi^1_{2n-2}$, and assume that the claim holds for
$\Psi$ and $\Psi^\circ$. Let $y_1$ and $y_2$ be the variables corresponding to the
normal occurrences $f_1(\bar z_{f_1})$ and $f_2(\bar z_{f_2})$, respectively.

We define the formula
\[
\Tot(\bar v) := \wforall \bar w\, \wnot\bigl((\bar v=\bar w)\to \bot \bigr),
\] 
where $\bar w$ is a tuple of fresh variables of the same length as $\bar v$.
If $X$ is nonempty, then $M,X\models \Tot(\bar v)$ iff  $X(\bar v)$ is all 
of $M^{|\bar v|}$.

Define
\begin{align*}
\Code_{f_1}&:=\Tot(\bar x,y_2,\ldots, y_{2n})\land\dep(\bar z_{f_1},y_1)
\qquad\text{and}\qquad  \\
\Code_{f_2}&:=\Tot(\bar x, y_3,\ldots,y_{2n})\land\dep(\bar z_{f_2},y_2),
\end{align*}
Put
\[
\Phi^\circ(\bar x,\bar y,\bar u)
 :=
 \Code_{f_1}\to  \wnot\bigl(\Code_{f_2}\to\wnot\Psi^\circ (\bar x,\bar y,\bar u)\bigr).
\]

First suppose $(M,\bar g)\models\Phi$, and let $X=M^{\bar x\bar y}[\bar G/\bar u]$, where $G(s)=g(s(\bar z_g))$.

Consider any subteam $Y\subseteq X$ such that
\[
M,Y\models\Code_{f_1}.
\]
If $Y=\emptyset$, there is nothing to prove by the empty team property. Otherwise, $Y$ codes a
total function $f_1:M^{|\bar z_{f_1}|}\to M$ such that
\[
Y=M^{\bar x,y_2,\ldots,y_{2n}}[\bar G/\bar u][F_1/y_1],
\]
where $F_1(s)=f_1(s(\bar z_{f_1}))$.
Since $(M,\bar g)\models\Phi$, there is a function
$f_2$ such that
\[
(M,\bar g,f_1,f_2)\models\Psi.
\]
Let $Z\subseteq Y$ be the subteam 
$M^{\bar x y_3,\ldots,y_{2n}}[\bar G/\bar u][F_1/y_1][F_2/y_2]$, where $F_2(s)=f_2(s(\bar z_{f_2}))$.
Then,
\[ M,Z \models \Code_{f_2}
\]
and, by the choice of $f_2$ and the induction hypothesis, 
\[ M,Z \models \Psi^{\circ}(\bar x, \bar y,\bar u).
\] 
Hence, $Y$ satisfies $\wnot\bigl(\Code_{f_2}\to\wnot\Psi^\circ(\bar x, \bar y,\bar u)\bigr)$.
Since $Y$ was arbitrary satisfying $\Code_{f_1}$, $M,X\models\Phi^\circ(\bar x, \bar y,\bar u)$.

Conversely, suppose $M,X\models\Phi^\circ(\bar x, \bar y,\bar u)$, where $X=M^{\bar x\bar y}[\bar G/\bar u]$, and fix an arbitrary interpretation of the function symbol $f_1$. Let   $Y=M^{\bar x,y_2,\ldots,y_{2n}}[\bar G/\bar u][F_1/y_1]$ as in the proof above.
Since $M,Y\models\Code_{f_1}$, the translation
yields a nonempty $Z\subseteq Y$ satisfying
\[
M,Z\models\Code_{f_2}\land\Psi^\circ(\bar x, \bar y,\bar u).
\]
The totality and dependence conditions ensure that $Z$ has exactly
the form above for some total function $f_2$, that is, $Z=M^{\bar x y_3,\ldots,y_{2n}}[\bar G/\bar u][F_1/y_1][F_2/y_2]$. The induction hypothesis
therefore gives $(M,\bar g,f_1,f_2)\models\Psi$.
As $f_1$ was arbitrary, $(M,\bar g)\models\Phi$.

This completes the induction. The statement of the theorem follows by taking $k=0$ and $\bar r=\bar x\bar y$.
\end{proof}

We can use the previous result to relate the complexity of the validity problem of $\FOTimp$ to that of second-order logic.

\begin{cor} The validity problem of $\FOTimp$ is equivalent to that of second-order logic.
\end{cor}
\begin{proof}  By \eqref{cccc} of Proposition~\ref{prop4.2}, $\FOTimp$ validity can be reduced to
second-order validity. Conversely, by Theorem~\ref{thm:sol} for any second-order sentence $\Phi$ it holds that
\[
\models \Phi
\quad\Longleftrightarrow\quad
\models
\Tot(\bar r)\wimp\Phi^\circ(\bar r).\qedhere
\]
\end{proof}

Theorem \ref{thm:sol}  allows us to also conclude that the usual universal quantifier of dependence logic is not expressible in $\FOTimp$.
Let us recall the semantics of the duplicating universal quantifier and the tensor disjunction.
\[
\begin{aligned}
    M,X \models \phi \vee \psi
        &\iff \textrm{exists $Y,Z$ s.t. $Y\cup Z=X$, }  M,Y \models \phi, \textrm{ and }  M,Z \models \psi, \\
   M,X \models \forall x\varphi
        &\iff M,X(M/x)\models\varphi .
\end{aligned}
\]
\begin{cor} The logic $\FOTimp$ is not semantically closed under the duplicating quantifier $\forall$ of dependence logic, i.e., it is not true that for every $\phi \in \FOTimp$ the formula $\forall x \phi$ is expressible in \FOTimp.
\end{cor}
\begin{proof}
There are sentences of the form $\forall \bar r \Phi^\circ$  
 that are not expressible in $\FOTimp$ as that would contradict  \eqref{aaaa} of Proposition \ref{prop4.2}.  
\end{proof}
It is worth noting that in dependence logic (in fact, in any team-based logic having the tensor disjunction $\vee$ and $\forall$) sentences give an  upper bound for the complexity of open formulas: given a formula $\phi(\bar x)$ we can define the sentence $$\phi^*:=\forall \bar x\big (\neg R(\bar x)\vee ( R(\bar x) \wedge \phi(\bar x))\big ).$$ Now for all $\M$ and teams $X$
$$\M,X\models \phi(\bar x) \iff (\M,\rel(X))\models \phi^*.  $$
Theorem \ref{thm:sol} shows that for  $\FOTimp$ we have a strong failure of this property.

\section{Discussion}

The results above illustrate two opposite phenomena. On the one hand, the
fragment $\FOTminus$ is very weak in the empty signature. Its formulas only record equality
patterns between variables, constancy of variables, and finite lower bounds on
the size of the model. In this sense, removing primitive disequality leaves very
little room for genuinely new team-semantical behaviour in the empty language.

On the other hand, adding intuitionistic implication changes the situation
dramatically. The contrast with the sentence case is worth emphasizing. By
Proposition~\ref{prop4.2}, intuitionistic implication does not increase expressive power on
sentences: over singleton teams, every $\FOT(\to)$-sentence is equivalent to its
ordinary first-order translation. Thus, at the level of sentences,
$\FOT(\to)$ remains first-order.

For open formulas, however, the situation changes completely. Intuitionistic
implication allows formulas to quantify over subteams of the current team. This
makes it possible to define dependence atoms, and hence to simulate
second-order quantification by considering suitable subteams of a full team. 
We can also view our results through the validity problem of the studied logics.

\begin{table}[ht]
\centering
\begin{tabular}{lcc}
\hline
 & $\tau=\emptyset$ & $\tau=\{ R(\cdot,\cdot)\}$ \\
\hline
$\FOTminus$ & Decidable & $\Sigma^0_1$-complete \\
$\FOT$     & $\Sigma^0_1$-complete & $\Sigma^0_1$-complete \\
$\FOTimp$  & SO-equivalent & SO-equivalent \\
\hline
\end{tabular}
\caption{Complexity of validity for all formulas.}
\label{tab:validity}
\end{table}

The decidability result follows from uniform quantifier elimination.
The effective translations into first-order logic give the
$\Sigma^0_1$ upper bounds. For $\FOTminus$ with a binary relation,
hardness follows from the translation of first-order sentences
$\alpha$ into their weak counterparts $\alpha^\ast$.

For $\FOT$ in the empty signature, the lower bound follows
from the close correspondence in expressive power between
$\FOT$ and first-order logic, using the team itself to encode
a relation.

Consequently, adding intuitionistic implication to the very weak team logic
$\FOTminus$ results in a logic which can express all second-order sentences
in the form of open formulas evaluated on full teams.


\section*{Funding}

The first author was supported by the Swedish Research Council
(Vetenskapsrådet) under grant 2022-01685 for the project ``Foundations for
team semantics: Meaning in an enriched framework''.  The second author was
partially supported by a Research Council of Finland grant No. 375116.

\section*{Declaration of AI use}

The authors used ChatGPT (OpenAI) for proofreading and identifying possible
gaps in arguments on the paper before submission. The authors reviewed all
suggestions and take full responsibility for the content of the article.

\bibliographystyle{abbrv}
\bibliography{refs}

\end{document}